\documentclass[11pt, twoside]{article}
\usepackage{amsmath,amsthm}
\usepackage{amsfonts}
\usepackage{amssymb,latexsym}
\usepackage{enumerate}
\usepackage{xcolor}
\usepackage{csquotes}
\usepackage{url}
\newtheorem{theorem}{Theorem}[section]

\newtheorem{proposition}[theorem]{Proposition}
\newtheorem{lemma}[theorem]{Lemma}
\newtheorem{definition}[theorem]{Definition}

\numberwithin{equation}{section}
\def\ind{1{\hskip -3 pt}\hbox{\textsc{I}}}
\def\n{\noindent}

\def\ve{\varepsilon}
\def\va{\varphi}
 
\def\e{\epsilon}
\def\o{\omega}
\def\O{\Omega}

\def\w{\wedge}
\def\cn{\mathbb C^n}

\def\R{\mathbb{R}}

\newcommand{\f}{\varphi}

\newcommand{\SH}{\mathcal{SH}}

\newcommand{\setdef}{\ \big\vert \ }
\newcommand{\Capa}{{\rm Cap}}

\newcommand{\SHXo}{\SH_m(X,\omega)}
\newcommand{\vep}{\varepsilon}

\newcommand{\EcXo}{\mathcal{E}(X,\omega,m)}

\newcommand{\Capm}{\Capa_{\omega,m}}

\begin{document}
\setlength{\baselineskip}{18truept}
\pagestyle{myheadings}

\markboth{   Mau Hai Le* and  Van Phu Nguyen**}{Convergence in $(\omega,m)$-capacity in the class $\mathcal{E}(X, \omega,m)$ }
\title {Convergence in $(\omega,m)$-capacity in the class $\mathcal{E}(X, \omega,m)$ on compact K\"ahler manifolds}
\author{
	 Mau Hai Le* and Van Phu Nguyen**
	\\ *Department of Mathematics, Hanoi National University of Education,\\ Hanoi, Vietnam.\\
	** Faculty of Natural Sciences, Electric Power University,\\ Hanoi,Vietnam;\\
	\\E-mail: mauhai@hnue.edu.vn and phunv@epu.edu.vn }

\date{}
\maketitle

\renewcommand{\thefootnote}{}



\renewcommand{\thefootnote}{\arabic{footnote}}
\setcounter{footnote}{0}

\begin{abstract}
In this paper, we establish the weak*-convergence of a sequence of the complex Hessian measures $H_m(u_j)$ to the complex Hessian measure $H_m(u)$ in the class $\mathcal{E}(X,\o,m)$ under hypotheses that $u_j$ is convergent to $u$ in $(\omega,m)$-capacity.
\end{abstract}

\section{Introduction}

 The complex Monge-Amp\`ere operator $(dd^c(\bullet))^n$ plays an important role in pluripotential theory and has been extensively studied by  many authors. This operator was first introduced and investigated by Bedford and Taylor \cite{BT76, BT82}. One of the main problems concerning this operator is to determine the conditions under which it is continuous in the weak*-topology. More precisely, given a sequence $\{u_j\}$ of bounded plurisubharmonic functions on an open subset $\O\subset \cn$, suppose that $u_j$ converges in some sense to a bounded plurisubharmonic function $u$. Under what conditions does $(dd^c u_j)^n$ converge to $(dd^c u)^n$ in the weak*-topology?\\  
 \n Bedford and Taylor \cite{BT82} showed that if $\{u_j\}_{j\geq 1}$ is decreasing to $u$ or increasing to $u$ almost everywhere with respect to the Lebesgue measure, then $(dd^c u_j)^n$ is convergent in the weak*-topology to $(dd^c u)^n$. Beyond the above-mentioned conditions, it is natural to ask whether weaker assumptions can ensure that $(dd^c u_j)^n$ converge to $(dd^c u)^n$ in the weak*-topology.  In \cite{Ce83}, by constructing an example, Cegrell showed that this operator is not continuous in the $L^1_{loc}$-topology. Concerning this problem,  Xing \cite{Xi96} introduced the notion of the convergence in $C_n$ and $C_{n-1}$-capacity of sequences $\{u_j\}_{j\geq1}$ of plurisubharmonic functions converging to a plurisubharmonic function $u$. Let $\O$ be an open subset of $\cn$ and let $\{u_j\}_{j\geq 1}$ be a sequence of plurisubharmonic functions on $\O$, with $u\in PSH(\O)$, where $PSH(\O)$ denotes the set of plurisubharmonic functions on $\O$. We say that the sequence $\{u_j\}$ converges to $u$ in $C_n$-capacity on $\O$(resp. $C_{n-1}$-capacity) if for all $\delta>0$ and all compact subsets $K\Subset\O$ we have
$$\lim\limits_{j\to \infty}C_n\Bigl(K\cap\{|u_j-u|>\delta\},\O\Bigl)=0,\eqno(1.1)$$

\n(resp. $\lim\limits_{j\to \infty}C_{n-1}\Bigl(K\cap\{|u_j-u|>\delta\},\O\Bigl)=0$), where $C_n$-capaccity (resp. $C_{n-1}$-capacity) was introduced and investigated in \cite{BT82} and \cite{Xi96}.  Xing \cite[Theorem 1]{Xi96} proved that if $\{u_j\}_j$ is a sequence of uniformly bounded plurisubharmonic functions that converges in $C_{n-1}$-capacity to a function  $u$ then $(dd^c u_j)^n$ converges to $(dd^c u)^n$ in the weak*-topology. \\
\n After that, U. Cegrell \cite{Ce98, Ce04} introduced several  function classes $\mathcal{F}(\Omega)$ and $\mathcal{E}(\Omega)$, which are not necessarily locally bounded, on which the complex Monge-Ampère operator is well-defined. Subsequently, Xing \cite[Theorem 3.4, Theorem 3.5]{Xi08b}, Cegrell \cite{Ce12} and L. M. Hai, P. H. Hiep, N. X. Hong, N. V, Phu \cite[Proposition 2.2]{HHHP14} extended the result in \cite{Xi96} from the class of bounded plurisubharmonic functions to the Cegrell class. \\
	B{\l}ocki \cite{Bl1} and Sadullaev–Abdullaev \cite{SA12} introduced the class of $m$-subharmonic functions as a natural extension of plurisubharmonic functions, along with the complex $m$-Hessian operator $H_m(.) = (dd^c.)^m \wedge \beta^{n-m}$, which generalizes the classical Monge–Amp\`ere operator $(dd^c.)^n$. Subsequently, in \cite{Ch12}, Chinh introduced the Cegrell-type classes $\mathcal{F}_m(\Omega)$ and $\mathcal{E}_m(\Omega)$, which are not necessarily locally bounded, and established that the complex $m$-Hessian operator is well defined on these classes. The relationship between convergence in $m$-capacity of $m$-subharmonic functions and the convergence of the corresponding Hessian operators was studied by the second author and collaborators in  papers \cite[Theorem 3.8]{HP17},  \cite[Corollary 2]{PDPu} and \cite[Corollary 3.3]{APD}.\\
\n In \cite{Xi09}, Xing studied the continuity of the complex Monge-Amp\`ere operator on compact K\"ahler manifolds. The main results in \cite{Xi09} were established for the class $\mathcal{E}(X,\o)$, which was introduced and investigated by Guedj and Zeriahi in \cite{GZ07}. For convenience to readers, now we recall some notions and results in \cite{Xi09} because our results in this note rely on ideals of \cite{Xi09};  However, we will work with the larger class $\mathcal{E}(X,\o,m)$.   Let $X$ be a compact K\"ahler manifold of dimension $n$ equipped with the fundamental form $\o$ given in local coordinates by
$$ \o=\frac{i}{2}\sum\limits_{\alpha,\beta} g_{\alpha,\bar{\beta}} dz^{\alpha}\wedge d\bar{z}^{\beta},$$

\n where $(g_{\alpha,\bar{\beta}})$ is a positive definite Hermitian matrix and $d\o=0$. The smooth volume form associated to the K\"ahler metric is given by the {\it n}th wedge product $\o^n$. By $PSH(X,\o)$ we denote the set of upper semi-continuous functions $u:X\longrightarrow \mathbb{R}\cup\{-\infty\}$ such that $u$ is integrable in $X$ with respect to the volume form $\o^n$ and $\o+dd^c u\geq 0$ in the current sense on $X$. As in \cite{BT87} for $u\in PSH(X,\o)$ we know that the complex Monge-Amp\`ere operator $(\o+dd^c u)^n$ is well-defined on the set $\{u>-\infty\}$. In \cite{GZ07}, Guedj and Zeriahi introduced the class $\mathcal{E}(X,\o)$. That is the subfamily of functions $u$ in $PSH(X,\o)$ such that $\int\limits_{\{ u>-\infty\}}(\o+dd^c u)^n=\int\limits_{X}\o^n$. We refer readers to \cite{GZ07} for important properties of this class. For brevity we use the notations $\o_u=\o+dd^c u$ and $\o_u^n= (\o+dd^c u)^n$. As in \cite{GZ07} we recall the notion of the Monge-Amp\`ere capacity $Cap_{\o}$ associated to $\o$ on a compact K\"ahler manifold $X$. For each a Borel $E\subset X$ the capacity of $E$ denoted by $Cap_{\o}(E)$ and is defined by
$$Cap_{\o}(E)=\sup\Biggl\{\int\limits_{E}\o_{u}^n : \  \ u\in PSH(X,\o),\  \ -1\leq u\leq 0.\Biggl\}$$

\n Next, we recall the notion of the convergence of a sequence $\{u_j\}_{j\geq 1}\subset PSH(X,\o)$ to a function $u\in PSH(X,\o)$ in the $Cap_{\o}$ on a compact K\"ahler manifold $X$. We say that a sequence $u_j$ in $PSH(X,\o)$ is called convegence in the capacity $Cap_{\o}$ to a function $u\in PSH(X,\o)$ if for every $\delta>0$ the following condition holds
$$\lim\limits_{j\to \infty}Cap_{\o}\Bigl(\{z\in X: |u_j(z)- u(z)|>\delta\}\Bigl)=0.$$

\n From this notion in \cite[Theorem 1]{Xi09}, Xing proved that if $u_j, u\in\mathcal{E}(X,\o)$ are such that $\{u_j\}$ converge to $u$ in $Cap_{\o}$ on $X$, then the complex Monge-Amp\`ere measures $(\o+dd^c u_j)^n$ converge to $(\o+dd^c u)^n$ weakly in $X$.  Note that as in an example of Cegrell mentioned above, the convergence in $L^1_{loc}$ of a sequence of plurisubharmonic functions $\{u_j\}$ to a plurisubharmonic function $u$ is not ensure that $(dd^c u_j)^n$ is weak*-convergent to $(dd^c u)^n$.  However, in \cite[Theorem 4 and Theorem 5 ]{Xi09}, the author added several additional assumptions to ensure that convergence in 
 $L^1(X)$ implies the convergence of the corresponding Monge–Ampère operator. Moreover, in \cite[Theorem 8]{Xi09}, under several additional assumptions, Xing proved that the convergence in $L^1(X)$ is equivalent to the convergence of the corresponding Monge–Ampère operator. Recently, Dinew-Chinh  \cite{DC15} and Chinh-Dong \cite{CD15}  introduced and investigated the class $\mathcal{E}(X,\o,m)$ which is strictly larger than the class $\mathcal{E}(X,\o)$ (see Theorem 5.12 and Example 5.13 in \cite{CD15}).  Building on this line of research, we generalize \cite[Theorem 8]{Xi09} by extending this result from class  $\mathcal{E}^1(X,\o)\subset \mathcal{E}(X,\o)$ to class $\mathcal{E}(X,\o,m).$ 

 \n  Our main result in this note are the following theorem. 
\begin{theorem}\label{2.4} Let $u_j,u\in\mathcal{E}(X,\omega, m)$. Assume that $H_m({u_j})\leq \mu$ and $H_m(u)\leq \mu$ where $\mu$ is a positive measure satisfying $\mu(X)=\int_{X}\o^n=1$ and $ \mu\ll \Capm$. Let also $\sup\limits_Xu_j=\sup\limits_Xu=0$. Then the following three statements are equivalent:
	
	i) $u_j\to u$ in $\Capm$;
	
	ii) $u_j\to u$ in $L^1(X,\o)$;
	
	iii) $H_m(u_j)\to H_m(u)$ weakly as $j\to\infty$.
\end{theorem}

\n The paper is organized as follows. Beside the introduction,  the paper has three sections. In Section 2 we recall the definitions and some basic results concerning to $(\o,m)$-subharmonic functions which were introduced and investigated Dinew-Chinh  in \cite{DC15} and Chinh-Dong in \cite{CD22}. In this section we deal with the class $\mathcal{E}(X,\o,m)$ on a compact K\"ahler manifold $X$ with a K\"ahler smooth metric $\o$.  We also recall the complex Hessian measure $H_m(\bullet)$ on $\mathcal{E}(X,\o,m)$ introduced in \cite{DC15} and \cite{CD22}. The section $\bf 2.3$ of this section is devoted to the presentation the notion of $(\o,m)$-capacity and the convergence in $(\o,m)$-capacity. The notion of a family of positive measures $\{\mu_{\alpha}\}$ is uniformly absolutely  continuous with respect to $(\o,m)$-capacity is also recalled in this section. The main results in this paper is contained in the third section. In this section we focus to prove Theorem \ref{2.4}.

\n{\bf Acknowledgement.} The work is supported from Ministry of Education and Training, Vietnam under the Grant number B2025-CTT-10.
This work is written in our visit in VIASM in the spring 2025. We also thank VIASM for financial support and hospitality.
 
\section{Preliminaries}

\subsection{$m$-subharmonic functions}
Let $M$ be a non-compact K\"ahler manifold of dimension $n$ and  $\omega$ be a K\"ahler form on $M$. Fix an integer $m$ such that $1\leq m\leq n$. We recall the definition of $m$-subharmonic functions for smooth functions as in \cite[Definition 2.1]{DC15}
\begin{definition}\label{def: m-sh smooth}
	A smooth function $u$ is called $m$-subharmonic ($m$-sh for short) on $M$ if the following condition holds in the classical sense
	$$
	(dd^c u)^k \wedge \omega^{n-k} \geq 0 , \ \forall k=1, \cdots, m .
	$$ 
	Equivalently, $u$ is $m$-sh if the vector of eigenvalues $\lambda(x) \in \R^n$ of $dd^c u$ with respect to $\omega$ satisfies
	$$
	S_k(\lambda(x)) \geq 0 , \ \forall x\in M,\ \forall k=1, \ldots, m .
	$$ 
	 
\end{definition}
\n For non-smooth functions, we have the following definition as in \cite[Definition 2.3]{DC15}
\begin{definition}\label{def: m-sh non-smooth}
	Let $u$ be an upper semi-continuous and locally integrable function on $M$.
	Then $u$ is called $m$-sh on $M$ if the following two conditions are satisfied:
	\begin{enumerate}[i)] 
		\item For any collection  $\varphi_1,\ldots,\varphi_{m-1}$ of smooth $m$-sh functions we have
		$$dd^c u\wedge dd^c \f_1\wedge \cdots \wedge dd^c \f_{m-1} \wedge \omega^{n-m}\geq 0, $$
		 in the weak sense of currents ;
		\item  if $v$ is an another function satisfying the above inequality and $v=u$ almost everywhere on $M$ then $u\leq v$. 
	\end{enumerate}
\end{definition}

Note that two $m$-sh functions are the same if they are equal almost everywhere on $M$. Observe also that by G{\aa}rding's inequality Definition \ref{def: m-sh smooth} and Definition \ref{def: m-sh non-smooth} are equivalent for smooth functions. The class of $m$-sh functions on $M$ (with respect to $\omega$) is denoted by 
$\SH_m(M)$.

\subsection{$(\omega,m)$-subharmonic functions}

\n Throughout this note, let $(X,\omega)$ be a compact K\"ahler manifold of complex dimension $n$, endowed with a K"ahler form $\omega = \omega_X$. We normalize so that $\int_X \omega^n = 1$, and fix an integer $m$ with $1 \leq m \leq n$. We recall the definition of $(\omega,m)$-subharmonic functions from \cite[Definition 2.6]{DC15}.

\begin{definition}
	A function $u: X\rightarrow \R\cup\{-\infty\}$ is called $(\omega,m)$-subharmonic ($(\omega,m)$-sh for short) if in any local chart $\Omega$ of $X$, the function 
	$\rho + u$ is $m$-sh, where $\rho$ is a local potential of $\omega$. 
\end{definition}
\n Observe that a smooth function $u$ is $(\omega,m)$-sh if and only if 
$$
(\omega+dd^c u)^k \wedge \omega^{n-k} \geq 0 ,\ \forall k=1,...,m.
$$

\n We denote the set of $(\omega, m)$-sh functions on $X$ by $\SH_m(X,\omega)$ and denote the set of negative $(\omega, m)$-subharmonic functions on $X$ by $\SH^-_m(X,\omega)$.
It follows from \cite[Theorem 1.2]{CD15} that for any $u\in \SHXo$ there exists a decreasing  sequence  of smooth $(\omega,m)$-sh functions on $X$ which converges to $u$. Following the classical pluripotential method of Bedford and Taylor \cite{BT76} one can then define the complex Hessian operator for bounded $(\omega,m)$-sh functions:
$$
H_m(u):=(\omega +dd^c u)^m\wedge \omega^{n-m},
$$
which is a non-negative (regular) Borel measure on $X$.\\
 Dinew and Chinh follow \cite{GZ07} to extend the definition of $H_m$ to unbounded $(\o,m)$-sh functions. Note that, we have
$$
\ind_{\{u>-j\}} H_m(\max(u,-j)) 
$$
is a non-decreasing sequence of positive Borel measures on $X$, for any $u\in \SHXo$, where $\ind_{\{u>-j\}}$ is the characteristic function of the set ${\{u>-j\}}$. Moreover, given any $(\omega,m)$-sh function $u$
$$\forall j\in \mathbb N,\ \int_X\ind_{\{u>-j\}} H_m(\max(u,-j))\leq\int_X\omega^n.$$

\n As in \cite[Definition 2.6]{DC15} and \cite[Definition 5.1]{CD15}, the class $\mathcal{E}(X,\o,m)$ is defined as the set of $(\omega,m)$-sh functions for  which 
$$
\lim_{j\rightarrow\infty}\int_X\ind_{\{u>-j\}} H_m(\max(u,-j))=\int_X\omega^n.
$$ For any $u\in \mathcal{E}(X,\o,m)$ set 
$$H_m(u):= \lim_{j\rightarrow\infty}\ind_{\{u>-j\}} H_m(\max(u,-j)).$$

\n Then as in the above definition we note that $u\in \mathcal{E}(X,\o,m)$ if and only if 
$$\int\limits_{\{u>-\infty\}} (\o+ dd^c u)^m\wedge\o^{n-m}= \int\limits_{X} \o^n.$$

\n By Lemma 5.2 in \cite{CD15} we see that a function $u\in\SH_m(X,\o)$ belongs to $\mathcal{E}(X,\o,m)$ if and only if $\lim\limits_{j\to\infty}\int_{\{u\leq-j\}}H_m(\max(u,-j))=0.$ .\\
\n  Proposition 5.10 in \cite{CD15} implies that $\EcXo$ is a convex set and it is stable under the max operation, that is, if $\varphi, \psi \in SH_m(X,\omega)$ are such that $\va$ or $\psi$ in $\EcXo$ then $\max{(\varphi,\psi)}\in \EcXo.$ 

\n We now recall the comparison principle and the maximum principle in $\EcXo$ which will be an important tool in the sequel (see Theorem 5.5 in \cite{CD15}). Note that, these results are rooted in Theorem 1.5 and Corollary 1.7 in \cite{GZ07} where analogous results for the class $\mathcal{E}(X,\o)$ may be found.
\begin{theorem}\label{comparison principle}
\n (i)	Let $\varphi,\psi \in {\mathcal E}(X,\omega,m)$, then we have 
$$\int_{\{\varphi<\psi\}}H_m(\psi) \leq \int_{\{\varphi<\psi\}} H_m(\varphi).$$
\n (ii) Let $\psi\in\EcXo,\varphi\in\SH_m(X,\o)$ then we have
	$$\ind_{\{\va<\psi\}}H_m(\max(\va,\psi))=\ind_{\{\va<\psi\}}H_m(\psi).$$
\end{theorem}
\n We also deal with the "partial comparison principle" as in \cite[Lemma 5.11]{CD15}.
\begin{theorem}\label{partial cp}
	Let T be a positive current of type
	$$T=(\o+dd^cu_1)\w\cdots\w(\o+dd^cu_k)\w\o^{n-m},\quad k<m,$$
	where $u_j\in\EcXo,j=1,\ldots,k.$ If $\varphi,\psi \in {\mathcal E}(X,\omega, m)$ then 
	$$\int_{\{\varphi<\psi\}}(\o+dd^c\va)^{m-k}\w T\leq \int_{\{\varphi<\psi\}}(\o+dd^c\psi)^{m-k}\w T.$$
	\end{theorem}
\n It follows from Corollary 2.8 in \cite{CD22} that we obtain "partial maximum principle".
\begin{theorem}\label{partial cmp}
		Let T be a positive current of type
	$$T=(\o+dd^cu_1)\w\cdots\w(\o+dd^cu_k)\w\o^{n-m},\quad k<m,$$
	where $u_j\in\EcXo,j=1,\ldots,k.$ If $\varphi,\psi \in {\mathcal E}(X,\omega,m)$ then
		$$\ind_{\{\va<\psi\}}\o^{m-k}_{(\max(\va,\psi))}\w T=\ind_{\{\va<\psi\}}\o^{m-k}_{\psi}\w T.$$ 
	\end{theorem}	
\subsection{$(\omega,m)$-capacity}
\n  Next we recall the notion of capacity  and the  notion of convergence in  $(\omega,m)$-capacity (see Definition 2.9 and Definition 2.10 in \cite{DC15}).
\begin{definition}
	The $(\omega,m)$-capacity of a Borel subset $E$ of $X$ is defined by 
	$$
	\Capm(E) : =\sup\left\{ \int_E H_m(u) \setdef u\in \SHXo , \ -1\leq u\leq 0\right\}.
	$$
\end{definition}
\begin{definition}
We say that	a sequence $\{u_j\}\subset SH_m(X,\o)$ converges to  $u\in SH_m(X,\o)$ in $\Capm$  if for any $\vep>0$
	$$
	\lim_{j\to +\infty} \Capm(\{|u_j-u|>\ve\}) =0.
	$$
\end{definition}
\n Exactly as for the class of plurisubharmonic functions we have convergence in $(\omega,m)$-capacity for monotonely sequences (See Proposition 2.11 in \cite{DC15}).
\begin{proposition}
	If $(u_j)\subset \SHXo$ is a monotone sequence of functions in $SH_m(X,\omega)$ which is convergent to $u \not \equiv -\infty$ then 
	$u_j$ converges to $u$ in $\Capm$.
\end{proposition}

\n According to Lemma 2.12 in \cite{DC15} we have the following proposition on the quasi-continuity of $(\omega,m)$-sh functions.
\begin{proposition}\label{quasicontinuity}
	Any $(\omega,m)$-sh function $u$ is quasi-continuous, i.e. for any $\vep>0$ there exists an open subset $U$ such that $\Capm(U)<\vep$ and 
	$u$ restricted on $X\setminus U$ is continuous.

\end{proposition}
	\n The following definition deals with the uniformly absolutely continuity with respect to $(\omega,m)$-capacity of a family of Borel measures.
\begin{definition} A family of positive measures $\{\mu_\alpha\}$ on $X$ is said to be uniformly absolutely continuous with respect to $(\omega,m)$-capacity if for every $\epsilon >0$ there exists $\delta >0$ such that for each Borel subset $E\subset X$ satisfying $\Capm(E)<\delta$ the inequality $\mu_\alpha (E)<\epsilon$ holds for all $\alpha$. We denote this by $\mu_\alpha\ll \Capm$.
\end{definition}
\n We recall Definition 4.8 in \cite{CD15} about $(\o,m)$-polar sets.
\begin{definition}
	A set $E\subset X$ is called $(\o,m)$-polar if for any $z\in E$ there exist neighborhood $V$ of $z$ and $v\in \SH_m(X,\omega)$ such that $E\cap V\subset\{v = -\infty\}$. 
\end{definition}
\n By Theorem 4.6, Lemma 4.9 and Theorem 4.10 in \cite{CD15} we see that  a Borel set E is $(\o,m)$-polar if and only if $Cap_{\o,m}(E)=0.$ Moreover, by the definition of the class ${\mathcal E}(X,\omega, m)$ we see that for all function $u\in {\mathcal E}(X,\omega,m)$ we have $H_m(u)$ puts  no mass on all $(\o,m)$-polar sets. \\
\n Note that, by Corollary 3.3 in \cite{CD15} we have $\mathcal{P}_m(X,\o)=\SHXo$. Thus, using Corollary 3.18 in \cite{Ch13a}, we get the following proposition.

\begin{proposition}\label{md2.12} Let $u\in \SHXo$. Then for $t> 0$
	$$Cap_{\o,m}(\{u<-t\})\leq \frac {(2m+1)|\sup\limits_{X} u|+m} t.$$
\end{proposition}
\n By Lemma 2.13 in \cite{DC15} we have a following result which is very useful in the next section.
\begin{proposition}\label{capacity} 
	Let $(\varphi_j)$ be an uniformly bounded sequence of functions in $\SH_m(X,\o)\cap L^{\infty}(X)$ converging in $\Capm$ to $\varphi.$ Then we have the weak*- convergence of measures $H_m(\varphi_j)\to H_m(\varphi).$
	\end{proposition}

\section{Convergence in $(\omega,m)$-capacity of  the class $\mathcal{E}(X, \omega,m)$}

Now, we give some results on the convergence in $(\omega,m)$-capacity in the class $\mathcal{E}(X,\omega,m)$.

\begin{theorem}\label{th3.1}
	If $u_j,\,u\in {\cal E}(X,\omega, m)$ are such that $u_j\to u$ in $\Capm$ on $X$, then
	$H_m(u_j)\to H_m(u)$ weakly in $X.$
\end{theorem}
\begin{proof}
	For any constant $k$ we write
	 \begin{equation}\label{eq3.1}\begin{aligned}
			H_m(u_j)-H_m(u)=&\big(H_m(u_j)-H_m({\max(u_j,-k)})\big)\\
		&+\bigl[H_m({\max(u_j,-k)})-H_m({\max(u,-k)})\bigr]\\
		&+\bigl[H_m({\max(u,-k)})-H_m(u)\bigr].
	\end{aligned}
	\end{equation}
	
	\n  It follows from Proposition \ref{capacity} that 
	\begin{equation}\label{eq3.1b}H_m({\max(u_j,-k_0)})- H_m({\max(u,-k_0)})\longrightarrow 0
		\end{equation} weakly on $X$ as $j\to\infty$ with $k_0$ is fixed.\\
	\n  Given a test function $\psi$, by Theorem \ref{comparison principle} we get that 
	\begin{equation}\label{e3.2}
	\begin{aligned}	
	&	\Bigl|\int_X\psi\,\bigl[H_m(u_j)-H_m({\max(u_j,-k)})\bigr]\Bigr|=\Bigl|\int_{\{u_j\leq -k\}}\psi\,\bigl[H_m(u_j)-H_m({\max(u_j,-k)})\bigr]\Bigr|\\
	&\leq \sup_X|\psi |\,\Bigl[\int_{\{u_j\leq -k\}}H_m(u_j)+\int_{\{u_j\leq -k\}}H_m({\max(u_j,-k)})\Bigr]\\
	&=\sup_X|\psi |\,\Bigl[\int_{\{u_j\leq -k\}}H_m(u_j)+\int_XH_m({\max(u_j,-k)})-\int_{\{u_j> -k\}}H_m({\max(u_j,-k)})\Bigr]\\
	&=\sup_X|\psi |\,\Bigl(\int_{\{u_j\leq -k\}}H_m(u_j)+\int_XH_m(u_j)
	-\int_{\{u_j> -k\}}H_m(u_j)\Bigr)\\
	&=2\,\sup_X|\psi |\,\int_{\{u_j\leq -k\}}H_m(u_j).
	\end{aligned}
\end{equation}
	\n Similarly, we also have 
$$\Bigl|\int_X\psi\,\bigl(H_m(\max(u,-k))-H_m(u)\bigr)\Bigr|\leq 2\,\sup_X|\psi |\,\int_{\{u\leq -k\}}H_m(u).$$
Since $u\in\EcXo,$ we have $\lim\limits_{k\to\infty}\int_{\{u\leq -k\}}H_m(u)=0. $
Thus, we have
\begin{equation}\label{eq3.3}
	\lim\limits_{k\to\infty}\Bigl|\int_X\psi\,\bigl(H_m(\max(u,-k))-H_m(u)\bigr)\Bigr|=0.
\end{equation}
From equality \eqref{eq3.1}, equality \eqref{eq3.1b}, inequality \eqref{e3.2} and equality \eqref{eq3.3}, it remains to prove that 
	\begin{equation}\label{eq3.4b}\lim_{k\to\infty}\limsup_{j\to\infty}\int_{u_j\leq -k}H_m(u_j)=0.
		\end{equation}
	Indeed, given $\varepsilon>0$ take $k_\varepsilon\geq 1$ such that 
	$\int_{ u\leq -k_\varepsilon+1}H_m(u)\leq \varepsilon/4$. We have 
	\begin{equation}\label{eq3.6a}
	\begin{aligned}
	&\int_{\{u_j\leq -k_\varepsilon\}}H_m(u_j)=\int_XH_m(u_j)-\int_{\{u_j>-k_\varepsilon\}}H_m(u_j)\\
	&=\int_X\omega^n-\int_{\{u_j>-k_\varepsilon\}}H_m(\max(u_j,-k_\varepsilon))\\
	&\leq \int_X\omega^n-\int\limits_{\{u_j>-k_\varepsilon\}\cap\{|u_j-u|
		\leq 1\}}H_m(\max(u_j,-k_\varepsilon))\\
	&\leq \int_X\omega^n-\int\limits_{\{u>-k_\varepsilon+1\}\cap\{|u_j-u|\leq 1\}}H_m(\max(u_j,-k_\varepsilon))\\
	&= \int_X\omega^n-\int_{\{u>-k_\varepsilon+1\}}H_m(\max(u_j,-k_\varepsilon))\\
	&+\int_{\{ u>-k_\varepsilon+1\}\cap\{|u_j-u|> 1\}}H_m(\max(u_j,-k_\varepsilon))\\
	&\leq \int_X\omega^n-\int_{\{u>-k_\varepsilon+1\}}H_m(\max(u_j,-k_\varepsilon))+\int_{\{|u_j-u|> 1\}}H_m(\max(u_j,-k_\varepsilon)).
	\end{aligned}
\end{equation}
	Since $u_j\to u$ in $\Capm$,  there exists $j_1$ such that $Cap_{\omega,m}\bigl(|u_j-u|> 1\bigr)\leq \varepsilon/4k_\varepsilon^m$ for $j\geq j_1$. Hence, we have
	\begin{equation}\label{eq3.6}
		\int_{\{|u_j-u|> 1\}}H_m(\max(u_j,-k_\varepsilon))\leq \frac{\varepsilon}{4}
		\end{equation}
	  for all $j\geq j_1$.\\
	  \n Moreover, by quasicontinuity of $(\o,m)$-sh functions as in Proposition \ref{quasicontinuity}, we can take a function $\bar u\in C(X)$ and an open subset $U$ such that $\Capm(U)\leq \varepsilon/4k_\varepsilon^m$ and $\bar u=u$ on $X\smallsetminus U.$ Hence we have
	  \begin{align*}
	  	&\int_{\{\bar u>-k_\varepsilon+1\}}H_m(\max(u_j,-k_\varepsilon))=\int_{\{\bar u>-k_\varepsilon+1\}\cap (X\smallsetminus U)}H_m(\max(u_j,-k_\varepsilon))\\
	  	&+\int_{\{\bar u>-k_\varepsilon+1\}\cap U}H_m(\max(u_j,-k_\varepsilon)) \\
	  & \leq \int_{\{ u>-k_\varepsilon+1\}\cap (X\smallsetminus U)}H_m(\max(u_j,-k_\varepsilon))+\int_{U}H_m(\max(u_j,-k_\varepsilon))\\
	  &  \leq \int_{\{ u>-k_\varepsilon+1\}}H_m(\max(u_j,-k_\varepsilon))+ {\varepsilon\over 4}.
	  \end{align*}
	\n This implies that
\begin{equation}\label{eq3.8}  
	\int_{\{ u>-k_\varepsilon+1\}}H_m(\max(u_j,-k_\varepsilon))\geq \int_{\{\bar u>-k_\varepsilon+1\}}H_m(\max(u_j,-k_\varepsilon))-{\varepsilon\over 4}.
	\end{equation}
Note that, by Proposition \ref{capacity} we have $H_m(\max(u_j,-k_\varepsilon))\longrightarrow H_m(\max(u,-k_\varepsilon))$ as $j\to\infty.$ Moreover, since $\bar{u}\in C(X),$ we have $\ind_{\{\bar u>-k_\varepsilon+1\}}$ is \textcolor{red}{a} lower semicontinuous function. According to Lemma 1.9 in \cite{De93} we have
$$ \liminf\limits_{j\to\infty}\int_{\{\bar u>-k_\varepsilon+1\}}H_m(\max(u_j,-k_\varepsilon))\geq \int_{\{\bar u>-k_\varepsilon+1\}}H_m(\max(u,-k_\varepsilon)).$$
Coupling this with inequality \eqref{eq3.8} we obtain
\begin{equation}\label{eq3.9}
	\liminf\limits_{j\to\infty}	\int_{\{ u>-k_\varepsilon+1\}}H_m(\max(u_j,-k_\varepsilon))\geq \int_{\{\bar u>-k_\varepsilon+1\}}H_m(\max(u,-k_\varepsilon))-{\varepsilon\over 4}.
\end{equation}
Now, from inequality \eqref{eq3.6a}, inequality \eqref{eq3.6}, inequality \eqref{eq3.9} and Theorem \ref{comparison principle}
 we obtain
	\begin{align*}&\limsup\limits_{j\to\infty}\int_{\{u_j\leq -k_\varepsilon\}}H_m(u_j)\leq \int_X\omega^n-\liminf\limits_{j\to\infty} \int_{\{\bar u>-k_\varepsilon+1\}}H_m(\max(u_j,-k_\varepsilon))+{\varepsilon\over 4}\\
	&\leq \int_XH_m(\max(u,-k_\varepsilon))- \int_{\{\bar u>-k_\varepsilon+1\}}H_m(\max(u,-k_\varepsilon))+{\varepsilon\over 2}\\
	&=\int_{\{\bar u\leq -k_\varepsilon+1\}}H_m(\max(u,-k_\varepsilon))+{\varepsilon\over 2}\\
	&=\int_{\{\bar u\leq -k_\varepsilon+1\}\cap (X\smallsetminus U)}H_m(\max(u,-k_\varepsilon)) + \int_{\{\bar u\leq -k_\varepsilon+1\}\cap U}H_m(\max(u,-k_\varepsilon)) + {\varepsilon\over 2}\\
	& \leq \int_{\{ u\leq -k_\varepsilon+1\}\cap (X\smallsetminus U)}H_m(\max(u,-k_\varepsilon)) + \int_{ U}H_m(\max(u,-k_\varepsilon)) + {\varepsilon\over 2}\\
	&\leq \int_{ \{u\leq -k_\varepsilon+1\}}H_m(\max(u,-k_\varepsilon))+{3\varepsilon\over 4}\\
	&= \int_{ \{u\leq -k_\varepsilon\}}H_m(\max(u,-k_\varepsilon))+\int_{ \{-k_\varepsilon<u\leq -k_\varepsilon+1\}}H_m(u)+{3\varepsilon\over 4}\\
	&= \int_X\omega^n-\int_{ \{u>-k_\varepsilon\}}H_m(u)+\int_{ \{-k_\varepsilon<u\leq -k_\varepsilon+1\}}H_m(u)+{3\varepsilon\over 4}\\
	&=\int_{ \{u\leq -k_\varepsilon+1\}}H_m(u)+{3\varepsilon\over 4}\leq \varepsilon,
	\end{align*}
	which yields inequality \eqref{eq3.4b}. The proof of theorem \ref{th3.1} is complete.
\end{proof}

\n We state a criterion for an uniformly absolutely continuity with respect to $(\o,m)$-capacity of Hessian measures of functions in $\mathcal{E}(X, \omega, m)$. 

\begin{proposition}\label{pro3.1} Let $u_j\in\mathcal{E}^-(X,\omega, m).$ If
	\begin{equation}\label{(*)}\lim\limits_{t\to +\infty}\varlimsup\limits_{j\to\infty} \int\limits_{\{u_j\leq -t\}} H_m(u_j)=0
	\end{equation}
	then $H_m(u_j)\ll \Capm$ uniformly for $j\geq 1$.
\end{proposition}
\n We need the following lemma which is independent interest.
\begin{lemma}\label{lm3.2}
	If $v\in\mathcal{E}(X, \omega, m),$ then $H_m(u)\ll\Capm$ on $X$ uniformly for all $u\in\SH_m(X,\o) $ with $0\geq u\geq v$ in $X.$
\end{lemma}
\begin{proof}
	Given $E\subset X$ and $u\in\SH_m(X,\o) $ with $0\geq u\geq v$ in $X.$ For each $k>0,$ by Theorem \ref{comparison principle} we have
	\begin{align*}
		\int_{E}H_m(u)&\leq \int_{u<-2k+2}H_m(u)+\int_{E\cap \{u>-2k\}}H_m(u)\\
		&\leq\int_{\{v<\frac{u}{2}-k+1\}}H_m(u)+\int_{E}H_m(\max(u,-2k))\\
		&\leq 2^m\int_{\{v<\frac{u}{2}-k+1\}}H_m(\frac{u}{2})+2^mk^m\Capm(E)\\
		&\leq 2^m\int_{\{v<\frac{u}{2}-k+1\}}H_m(v)+2^mk^m\Capm(E)\\
		&\leq 2^m\int_{\{v<-k+1\}}H_m(v)+2^mk^m\Capm(E).
	\end{align*}
Note that, since $v\in\mathcal{E}(X, \omega,m)$  we have $\lim\limits_{k\to\infty}\int_{\{v<-k+1\}}H_m(v)=0.$ Therefore, for $\varepsilon>0,$ there exist $k_0$ such that $\int_{\{v<-k_0+1\}}H_m(v)\leq \frac{\varepsilon}{2^{m+1}}.$ We choose $\delta=\frac{\varepsilon}{2^{m+1}{k^m_0}}.$ Obviously, if $\Capm(E)\leq\delta$ then $\int_{E}H_m(u)\leq\varepsilon.$
This yields that $H_m(u)\ll \Capm $ on $X$ uniformly for all such functions $u.$ The proof of Lemma \ref{lm3.2} is complete.
	\end{proof}
\n Now, we will prove Proposition \ref{pro3.1}.
\begin{proof}    Fix $\e>0$. By equation (\ref{(*)}), we choose $t_0$ and then $j_0=j_0(t_0)>1$ such that
	$$\int\limits_{\{u_j\leq -t_0\}}H_m(u_j)<\e,$$
	for all $j\geq j_0$. For each Borel set $E\subset X,$ by Theorem \ref{comparison principle} we can estimate
	\begin{align*}\int\limits_E H_m(u_j)&=\int\limits_{E\cap\{u_j\leq -t_0\}}H_m(u_j)+\int\limits_{E\cap\{u_j>-t_0\}}H_m(u_j)\\
		&\leq\int\limits_{\{u_j\leq -t_0\}}H_m(u_j)+\int\limits_{E\cap\{u_j>-t_0\}} H_m(\max (u_j,-t_0))\\
		&\leq\int\limits_{\{u_j\leq -t_0\}}H_m(u_j)+t_0^m\Capm(E)\leq\e+t_0^m\Capm(E)
	\end{align*}
	for all $j\geq j_0$.\\
	 On the other hand, by Lemma \ref{lm3.2}, for each $k\in\{1,2,\cdots,j_0\}$ we can choose $\delta_k>0$ such that $H_m(u_k)(E)<\e$  for all Borel sets $E\subset X$ with $\Capm(E)<\delta_k$. Choose $\delta^{'}=\min(\delta_1,\ldots,\delta_{j_0}).$  This implies that $H_m(u_k) (E)<\e$ for all Borel sets $E\subset X$ with $\Capm(E)<\delta^{'}$ for every $k\in\{1,2,\ldots,j_0\}.$
	  Hence
	$H_m(u_j) (E)<2\e$ for all $j\geq 1$ and all  Borel sets $E\subset X$, such that $\Capm(E)<\delta=\min (\delta^{'},\frac {\e}{t_0^m})$. Thus, we obtain $H_m(u_j)\ll \Capm$ uniformly for $j\geq 1$. The proof of Proposition \ref{pro3.1} is complete.
\end{proof}

\n We state the following lemma which is independent interest.
\begin{lemma}\label{lm3.4}
	Let $u_j,v_j\in\mathcal{E}^{-}(X, \omega,m)$ be such that $u_j\geq v_j$ for all $j.$ If  $H_m(v_j)\ll\Capm$  uniformly for all $j$ and $\inf_{j\geq 1}\sup_{X}v_j>-\infty$ then  $H_m(u_j)\ll\Capm$  uniformly for all $j$
\end{lemma}

\begin{proof}
	
	\n By Proposition \ref{md2.12} we have $$Cap_{\o,m}(\{v_j<-t\})\leq \frac {\big[(2m+1)|\sup\limits_X v_j|+m\big]} t.$$ Put the above inequality together with the assumption $\inf_{j\geq 1}\sup_{X}v_j>-\infty$ and  $H_m(v_j)\ll\Capm$  uniformly for all $j$ we infer that
	$\lim\limits_{t\to +\infty}\varlimsup\limits_{j\to\infty} \int\limits_{\{v_j\leq -t\}} H_m(v_j)=0.$ 
	On the other hand, by Theorem \ref{comparison principle}, we have 
	\begin{align*}\int_{\{u_j<-2t\}}H_m(u_j)&\leq \int_{\{v_j<\frac{u_j}{2}-t\}}H_m(u_j)\leq 2^m\int_{\{v_j<\frac{u_j}{2}-t\}}H_m(\frac{u_j}{2})\\
		&\leq 2^m\int_{\{v_j<\frac{u_j}{2}-t\}}H_m(v_j)\leq 2^m\int_{\{v_j<-t\}}H_m(v_j).
	\end{align*}
	Therefore, we deduce that 	$\lim\limits_{t\to +\infty}\varlimsup\limits_{j\to\infty} \int\limits_{\{u_j\leq -2t\}} H_m(u_j)=0.$ According to Proposition \ref{pro3.1}, we imply that $H_m(u_j)\ll\Capm$  uniformly for all $j.$ The proof of Lemma \ref{lm3.4} is complete.
\end{proof}

\begin{theorem}\label{th3.5}
Let $u_j\in\EcXo, u\in \SH^{-}_m(X,\o)$ be such that $u_j\to u$ in $\Capm.$ Then the following statements are equivalent:\\
\n (i) $u\in\EcXo$;\\

\n (ii) $H_m(u_j)\ll \Capm$ uniformly for $j\geq 1.$
\end{theorem}
\begin{proof}
	(i) $\Rightarrow$ (ii) According to Proposition \ref{pro3.1}, it is enough to prove that 
	$$\lim\limits_{t\to +\infty}\varlimsup\limits_{j\to\infty} \int\limits_{\{u_j\leq -t\}} H_m(u_j)=0.$$
	Indeed, it is easy to check that
	$$\{u_j\leq -t\}\subset \{u\leq -t+1\}\cup \{|u_j-u|>1\}.$$
	 By Proposition \ref{comparison principle} we have
	 \begin{equation}\label{e3.11}
	\begin{aligned}&\int\limits_{\{u_j\leq -t\}} H_m(u_j)= \int\limits_{X} H_m(u_j)-\int\limits_{\{u_j> -t\}} H_m(u_j)\\
	&=\int\limits_{X}\o^n-\int\limits_{\{u_j>-t\}} H_m(\max{(u_j,-t)})\\
	&=\int\limits_{X}H_m(\max{(u_j,-t)})-\int\limits_{\{u_j>-t\}} H_m(\max{(u_j,-t)})\\
	 &=\int\limits_{\{u_j\leq -t\}} H_m(\max{(u_j,-t)})\\
	 &\leq\int\limits_{\{u\leq -t+1\}} H_m(\max (u_j,-t))+\int\limits_{\{|u_j-u|>1\}} H_m(\max (u_j,-t))\\
	&\leq\int\limits_{\{u\leq -t+1\}} H_m(\max (u_j,-t))+t^m\Capm(\{|u_j-u|>1\}).
	\end{aligned}
\end{equation}

\n Fix $\e>0.$ Since $\max(u,-t)\geq u\in\EcXo$ for all $t,$ according to Lemma \ref{lm3.2} we have $H_m(\max(u,-t))\ll \Capm$ for every $t.$ Hence, there exists $\delta>0$ such that $H_m(\max(u,-t))(E)<\frac{\e}{3}$ for all Borel set $E\subset X$ satisfying $\Capm(E)<\delta.$  Since $\lim\limits_{t\to\infty}\Capm(\{u<-t+1\})=0,$ there exists $t_0$ such that $\Capm(\{u<-t_0+1\})<\delta.$ This implies that 
\begin{equation}\label{eq3.12b}H_m(\max(u,-t_0))(\{u<-t_0+1\}) <\frac{\e}{3}.
	\end{equation}
By Propostion \ref{quasicontinuity} there exist $\bar{u}\in C(X)$ and an open subset $U$ such that \begin{equation}\label{eq3.12c}\Capm(U)\leq\frac{\e}{3t^m_0},\end{equation}

\n and $\bar{u}= u $ on $X\setminus U$.
We have 
\begin{equation}\label{eq3.12}
	\begin{aligned}
		&\int\limits_{\{u\leq -t+1\}} H_m(\max (u_j,-t))\\
		&=\int\limits_{\{u\leq -t+1\}\cap (X\smallsetminus U)} H_m(\max (u_j,-t)) +\int\limits_{\{u\leq -t+1\}\cap U} H_m(\max (u_j,-t))\\
		&\leq\int\limits_{\{\bar{u}\leq -t+1\}\cap (X\smallsetminus U)} H_m(\max (u_j,-t))+\int\limits_{ U} H_m(\max (u_j,-t))\\
		&\leq \int\limits_{\{\bar{u}\leq -t+1\}\cap (X\smallsetminus U)} H_m(\max (u_j,-t))+t^m\Capm(U).
	\end{aligned}
\end{equation}
Note that according to Proposition \ref{capacity} we have $H_m(\max (u_j,-t))\to H_m(\max (u,-t))$ weakly as $j\to\infty.$ Moreover, $\ind_{\{\{\bar{u}\leq -t+1\}\cap (X\smallsetminus U) \}}$  is upper semicontinuous. Thus, by Lemma 1.9 in \cite{De93} we obtain 
$$ \lim\limits_{j\to\infty} \int\limits_{\{\bar{u}\leq -t+1\}\cap (X\smallsetminus U)} H_m(\max (u_j,-t))\leq \int\limits_{\{\bar{u}\leq -t+1\}\cap (X\smallsetminus U)} H_m(\max (u,-t)).$$
Therefore, there exists $j_1$ such that 
$$\int\limits_{\{\bar{u}\leq -t_0+1\}\cap (X\smallsetminus U)} H_m(\max (u_j,-t_0))\leq \int\limits_{\{\bar{u}\leq -t_0+1\}\cap (X\smallsetminus U)} H_m(\max (u,-t_0)) $$ for all $j\geq j_1.$\\
On the other hand, since $u_j\to u$ in $\Capm.$ There exists $j_2$ such that we have
$$ \Capm(\{|u_j-u|>1\})\leq \frac{\e}{3t^m_0}$$ for all $j\geq j_2.$\\
Hence, for all $j\geq j_0=\max(j_1,j_2),$
it follows from inequality \eqref{eq3.12} that 
\begin{equation}\label{eq3.13}
	\begin{aligned}&\int\limits_{\{u\leq -t_0+1\}} H_m(\max (u_j,-t_0))\\
		&\leq \int\limits_{\{\bar{u}\leq -t_0+1\}\cap (X\smallsetminus U)} H_m(\max (u,-t_0)) +t_0^m\Capm(U)\\
		&\leq \int\limits_{\{u\leq -t_0+1\}\cap (X\smallsetminus U)} H_m(\max (u,-t_0)) +t^m_0\Capm(U)\\
		& \leq \int\limits_{\{u\leq -t_0+1\}} H_m(\max (u,-t_0)) +t^m_0\Capm(U)\\
		&\leq \frac{2\e}{3},
	\end{aligned}
	\end{equation}
where the last inequality follows from inequality \eqref{eq3.12b} and  inequality \eqref{eq3.12c}.
\n Coupling inequality \eqref{e3.11} and inequality \eqref{eq3.13}, we get 
$$\int_{u_j\leq  t_0}H_m(u_j)\leq \e $$ for all $j\geq j_0.$\\
	 This means we have 
	$$\lim\limits_{t\to +\infty}\varlimsup\limits_{j\to\infty}\int\limits_{\{u_j\leq -t\}} H_m(u_j)=0.$$
	Hence, by Proposition \ref{pro3.1} we infer that $H_m(u_j)\ll \Capm$ uniformly for $j\geq 1$ and the desired conclusion follows.
	
\n	(ii) $\Rightarrow$ (i) For any fixed $k>0,$ by Proposition \ref{capacity} we have  $H_m(\max(u_j,-k))\to H_m(\max(u,-k))$ weakly as $j\to\infty$. Moreover, we also have $\ind_{\{u<-k+1\}}$ is a lower semicontinuous function. Hence, by Lemma 1.9 in \cite{De93} we obtain
	\begin{align*}&\int_{u<-k+1}H_m(\max(u,-k))\leq \lim\limits_{j\to\infty}\int_{u<-k+1}H_m(\max(u_j,-k))\\
		&\leq \limsup\limits_{j\to\infty}\int_{u_j<-k+{5\over 4}}H_m(\max(u_j,-k)) +\limsup\limits_{j\to\infty}\int_{|u_j-u|>{1\over 4}}H_m(\max(u_j,-k))\\
		&\leq \limsup\limits_{j\to\infty}\int_{u_j<-k+{5\over 4}}H_m(\max(u_j,-k)) +\limsup\limits_{j\to\infty}k^m\Capm\Bigl({|u_j-u|>{1\over 4}}\Bigr)\\
		&=\limsup\limits_{j\to\infty}\int_{u_j<-k+{5\over 4}}H_m(\max(u_j,-k))\\
		&=\limsup\limits_{j\to\infty}\Bigl(\int_XH_m(\max(u_j,-k))-\int_{u_j\geq -k+{5\over 4}}H_m(\max(u_j,-k))\Bigr)\\
		&=\limsup\limits_{j\to\infty}\Bigl(\int_XH_m(u_j)-\int_{u_j\geq -k+{5\over 4}}H_m(u_j)\Bigr)\\
	&=\limsup\limits_{j\to\infty}\int_{u_j< -k+{5\over 4}}H_m(u_j)\\
	&\leq \limsup\limits_{j\to\infty}\int_{u< -k+{3\over 2}}H_m(u_j)+\limsup\limits_{j\to\infty}\int_{|u_j-u|>{1\over 4}}H_m(u_j).
	\end{align*}
On one hand, according to the hypothesis $H_m(u_j)\ll \Capm$, for every $\varepsilon>0,$ there exists $\delta>0$ such that for every Borel subset $E\subset X$ satisfies $\Capm(E)<\delta,$ we have $H_m(u_j)(E)<\varepsilon$ for all $j.$ We also have $\lim\limits_{k\to\infty}\Capm(\{u< -k+{3\over 2}\})=0.$ Hence,  there exists $k_0\in\mathbb{N}$ such that for all $k\geq k_0,$ we have $ \Capm(\{u< -k+{3\over 2}\})<\delta.$ This implies that $H_m(u_j)(\{u< -k+{3\over 2}\})<\varepsilon$ for all $j.$ Thus, we obtain that $\lim\limits_{k\to\infty}\limsup\limits_{j\to\infty} H_m(u_j) (\{u< -k+{3\over 2}\})=0.$
\\On the other hand, it follows from $H_m(u_j)\ll \Capm$ and $u_j\to u$ in $\Capm$ as $j\to\infty$ that 
$$\limsup\limits_{j\to\infty}\int_{|u_j-u|>{1\over 4}}H_m(u_j)=0. $$
Therefore, we have $\lim\limits_{k\to\infty}\int_{u<-k+1}H_m(\max(u,-k))=0.$
	According to Lemma 5.2 in \cite{CD15} we get $u\in {\cal E}(X,\omega,m)$. The proof of Theorem \ref{th3.5}  is complete.
\end{proof}

\n Combine Theorem \ref{th3.1} and Theorem \ref{th3.5} we get the following theorem.
\begin{theorem}
	 Let $u\in \SH^{-}_m(X,\omega)$. Suppose that there exists a sequence $u_j \in \EcXo$ is such that $u_j\to u$ in $\Capm$ on $X$, and $H_m(u_j)\ll \Capm$ on $X$ uniformly for all $j$. Then $u\in {\cal E}(X,\omega,m)$ and $H_m(u_j)\to H_m(u)$ weakly in $X$.
\end{theorem}

\n Now, we give an important technical result by the following proposition.

\begin{proposition}\label{pro3.7} Let $u_j,v_j\in\mathcal{E}(X,\omega,m)$ be such that
	
	i) $H_m(u_j)\ll \Capm$ uniformly for $j\geq 1$;
	
	ii) $\lim\limits_{j\to \infty} \int\limits_{\{u_j<v_j-\delta\}}H_m(u_j)=0$ for all $\delta >0$.
	
	\n Then $\lim\limits_{j\to\infty}\Capm (\{u_j<v_j-\delta\})=0$ for all $\delta >0$.
\end{proposition}

\begin{proof} We can assume, by adding constants to $u_j$ if necessary, that $\sup\limits_X u_j =0$ for all $j\geq 1$. Set $u_{jt}:=\max (u_j, -t)$. For each $k=0,\ldots,m$ we will prove inductively that for every $\delta >0$
	\begin{equation}\label{1}\lim\limits_{j\to\infty}\sup\Bigl\{\int\limits_{\{u_j<v_j-\delta\}}\omega_{u_j}^{m-k}\w \omega_{\va}^k\wedge\omega^{n-m}:\ \va\in SH_m(X,\omega),\ -1\leq\va\leq 0\Bigl\}=0.
	\end{equation}
	
	\n Assume that (3.16) holds. Then set $k=m$ we infer that $\lim\limits_{j\to \infty} Cap_{\omega,m}(\{u_j< v_j-\delta\})=0$ and the desired conclusion follows.
	Note that if $k=0$ then (\ref{1}) holds from the assumption (ii). Assume that (\ref{1}) holds for $k-1$. We will prove that
	$$\lim\limits_{j\to\infty}\sup\Bigl\{\int\limits_{\{u_j<v_j-3\delta\}}\omega_{u_j}^{m-k}\w\omega_{\va}^k\wedge\omega^{n-m}:\ \va\in SH_m(X,\omega),\ -1\leq\va\leq 0\Bigl\}=0$$
	for any $\delta >0$.\\
	\n  We fix $t\geq 1$ and $\va\in SH_m (X,\omega)$ satisfying $-1\leq\va\leq 0$. Let $\eta_{j},\ \beta_{j}$ be the $(\omega, m)$-subharmonic functions defined respectively  by $\eta_j:=\frac{v_j+\frac {\delta}{t}\va -2\delta}{1+\frac {\delta} t},\ \beta_j:=\frac{u_j+\frac {\delta}{t}u_{jt}}{1+\frac {\delta} t}$ and put $U=\{\beta_j<\eta_j\}$. Note that we have $\beta_j,\eta_j\in \EcXo $ since $\EcXo$ is a convex set. It is easy to check that 
	$$\{u_j<v_j-3\delta\}\subset U\subset\{u_j<v_j-\delta\}.$$
	Then we have
	
	\begin{equation}\label{e3.3} \int\limits_{\{u_j<v_j-3\delta\}}\omega_{u_j}^{m-k}\w\omega_{\va}^k\wedge\omega^{n-m}\leq\int\limits_{U}\omega_{u_j}^{m-k}\w\omega_{\va}^k\wedge\omega^{n-m}.
	\end{equation}
	Moreover, we also have
	$$\omega_{u_j}^{m-k}\w\omega_{\va}^{k-1}\wedge\omega^{n-m}\w[\frac {t+\delta} {\delta}\omega_{\eta_j}-\o_{\va}]=\frac {t} {\delta}\o_{v_j}\w\omega_{u_j}^{m-k}\w\omega_{\va}^{k-1}\wedge\omega^{n-m}\geq 0.$$
	This implies that
	\begin{equation}\label{e3.4}
	\int_{U}\omega_{u_j}^{m-k}\w\omega_{\va}^k\wedge\omega^{n-m}\leq \int_{U}	\frac {t+\delta} {\delta}\omega_{\eta_j}\w\omega_{u_j}^{m-k}\w\omega_{\va}^{k-1}\wedge\omega^{n-m}.
	\end{equation}

\n By applying Theorem \ref{partial cp} and Theorem \ref{partial cmp}  we obtain 
	\begin{align*}
	&\frac {t+\delta} {\delta}\int\limits_{U}\omega_{\eta_j}\w\omega_{u_j}^{m-k}\w\omega_{\va}^{k-1}\wedge\omega^{n-m}\leq
		\frac {t+\delta} {\delta}\int\limits_{U}\omega_{\beta_j}\w\omega_{u_j}^{m-k}\w\omega_{\va}^{k-1}\wedge\omega^{n-m}\\
	&\leq\frac {t} {\delta}\int\limits_{\{u_j < v_j-\delta\}}\omega_{u_j}^{m-k+1}\w\omega_{\va}^{k-1}\wedge\omega^{n-m}
	+\int\limits_{\{u_j < v_j-\delta\}}\omega_{u_{jt}}\w\omega_{u_j}^{m-k}\w\omega_{\va}^{k-1}\wedge\omega^{n-m}.\\ 
	&\leq \frac {t} {\delta}\int\limits_{\{u_j < v_j-\delta\}}\omega_{u_j}^{m-k+1}\w\omega_{\va}^{k-1}\wedge\omega^{n-m}
		+\int\limits_{\{u_j < v_j-\delta\}}\omega_{u_j}^{m-k+1}\w\omega_{\va}^{k-1}\wedge\omega^{n-m}\\
	&+ \int\limits_{\{u_j < v_j-\delta\}}\omega_{u_{jt}}\w\omega_{u_j}^{m-k}\w\omega_{\va}^{k-1}\wedge\omega^{n-m} -\int\limits_{\{u_j < v_j-\delta\}}\omega_{u_j}^{m-k+1}\w\omega_{\va}^{k-1}\wedge\omega^{n-m}\\
	& \leq \frac {t+\delta} {\delta}\int\limits_{\{u_j < v_j-\delta\}}\omega_{u_j}^{m-k+1}\w\omega_{\va}^{k-1}\wedge\omega^{n-m}\\
	&+ \int\limits_{\{u_j < v_j-\delta\}\cap \{u_j\leq-t\}}\omega_{u_{jt}}\w\omega_{u_j}^{m-k}\w\omega_{\va}^{k-1}\wedge\omega^{n-m}\\
	&-\int\limits_{\{u_j < v_j-\delta\}\cap\{u_j\leq -t\}}\omega_{u_j}^{m-k+1}\w\omega_{\va}^{k-1}\wedge\omega^{n-m}\\
	& \leq \frac {t+\delta} {\delta}\int\limits_{\{u_j < v_j-\delta\}}\omega_{u_j}^{m-k+1}\w\omega_{\va}^{k-1}\wedge\omega^{n-m} +  \int\limits_{ \{u_j\leq-t\}}\omega_{u_{jt}}\w\omega_{u_j}^{m-k}\w\omega_{\va}^{k-1}\wedge\omega^{n-m},
	\end{align*} where the fourth inequality follows from the fact that
$$\ind_{\{u_j>-t\}}\omega_{u_{jt}}\w\omega_{u_j}^{m-k}\w\omega_{\va}^{k-1}\wedge\omega^{n-m}= \ind_{\{u_j>-t\}}\omega_{u_j}^{m-k+1}\w\omega_{\va}^{k-1}\wedge\omega^{n-m}.$$
	Combine the above equality, inequality \eqref{e3.3}, inequality \eqref{e3.4} and the induction hypothesis, we get
	
	\begin{align*}&\lim\limits_{j\to\infty}\sup\Bigl\{\int\limits_{\{u_j<v_j-3\delta\}}\omega_{u_j}^{m-k}\w\omega_{\va}^k\wedge\omega^{n-m}:\ \va\in SH_m(X,\omega),\ -1\leq\va\leq 0\Bigl\}\\
		&\leq \lim\limits_{j\to\infty}\sup\Bigl\{\int\limits_{\{u_j\leq -t\}}\omega_{u_{jt}}\w\omega_{u_j}^{m-k}\w\omega_{\va}^{k-1}\wedge\omega^{n-m}:\ \va\in SH_m (X,\omega),\ -1\leq\va\leq 0\Bigl\},\end{align*}
	
\n	for every $t\geq 1$.\\
	\n Note that since $\sup_{X}u_j=0$ for all $j\geq 1$ and $-1\leq \varphi\leq 0,$ we have  $$\tilde{u}_j:= \frac{u_{jt}+u_j+\va}{3}\geq \frac{2u_j-1}{3}\geq \frac{3u_j-1}{3} = u_j-\frac{1}{3}.$$ Therefore, according to the hypothesis (i) and Lemma \ref{lm3.4} we get that $H_m(\tilde{u}_j)\ll \Capm$ uniformly for $j\geq 1$. Hence, so is $3^m H_m(\tilde{u}_j)$. But it is not difficult to see that
	$$3^m H_m(\tilde{u}_j)\geq\omega_{u_{jt}}\w\omega_{u_j}^{m-k}\w\omega_{\va}^{k-1}\wedge\omega^{n-m}.$$
	
	\n It follows that $\omega_{u_{jt}}\w\omega_{u_j}^{m-k}\w\omega_{\va}^{k-1}\wedge\omega^{n-m}<< Cap_{\omega, m}$ uniformly for $j\geq 1$.
	
\n By Proposition \ref{md2.12} we get that 

$$Cap_{\omega,m}(\{u_j\leq -t\})\leq\frac{(2m+1)|\sup\limits_{X}u_j|+m}{t} = \frac{m}{t},$$

\n for all $j\geq 1$. Since $\omega_{u_{jt}}\w\omega_{u_j}^{m-k}\w\omega_{\va}^{k-1}\wedge\omega^{n-m}<< Cap_{\omega, m}$ uniformly for $j\geq 1$ then for $\varepsilon >0$ there exists $\delta>0$ such that for all Borel subset $E\subset X$ with $Cap_{\omega,m}(E)<\delta$ it follows that $\omega_{u_{jt}}\w\omega_{u_j}^{m-k}\w\omega_{\va}^{k-1}\wedge\omega^{n-m}(E)<\varepsilon$ for all $j\geq 1$. Using this we can choose $t_0>0$ such that for $t\geq t_0$ we get that  $Cap_{\omega,m}(\{u_j\leq -t\})<\delta$ for all $j\geq 1$. Hence, 
$$\int\limits_{\{u_j\leq -t\}}\omega_{u_{jt}}\w\omega_{u_j}^{m-k}\w\omega_{\va}^{k-1}\wedge\omega^{n-m}\leq\varepsilon,$$

\n for all $j\geq 1$ and $-1\leq \varphi\leq 0$. This yields that
$$\sup\Bigl\{\int\limits_{\{u_j\leq -t\}}\omega_{u_{jt}}\w\omega_{u_j}^{m-k}\w\omega_{\va}^{k-1}\wedge\omega^{n-m}:-1\leq\varphi\leq 0,\Bigl \}\leq\varepsilon,$$

\n for all $j\geq 1$. Therefore, 
$$\lim\limits_{j\to\infty}\sup\Bigl\{\int\limits_{\{u_j\leq -t\}}\omega_{u_j}^{m-k}\w\omega_{\va}^k\wedge\omega^{n-m}:\ \va\in SH_m(X,\omega),\ -1\leq\va\leq 0\Bigl\}=0$$
	as $t \to\infty.$ Thus, we have
	$$\lim\limits_{j\to\infty}\sup\Bigl\{\int\limits_{\{u_j<v_j-3\delta\}}\omega_{u_j}^{m-k}\w\omega_{\va}^k\wedge\omega^{n-m}:\ \va\in SH_m(X,\omega),\ -1\leq\va\leq 0\Bigl\}=0.$$ Thus the proof of Proposition \ref{pro3.7} is complete.
\end{proof}

\n As a direct corollary we obtain the following result.

\begin{theorem}\label{2.2} Let $u_j,v_j\in\mathcal{E}(X,\omega,m)$ be such that
	
	i) $H_m({u_j})+H_m(v_j)\ll \Capm$ uniformly for $j\geq 1$;

	ii) $\lim\limits_{j\to \infty}[ \int\limits_{\{u_j<v_j-\delta\}}H_m(u_j)+\int\limits_{\{v_j<u_j-\delta\}}H_m(v_j)]=0$ for all $\delta >0$.
	
	\n Then $u_j - v_j \to 0$ in $\Capm$.
\end{theorem}

\n Our next result provides such a characterization in the class $\mathcal{E}(X,\omega, m)$.

\begin{theorem}\label{2.3} Let $u_j\in\mathcal{E}(X,\omega,m)$ and $u\in \SH_m (X,\omega)$. Then the following three statements are equivalent:
	
	i) $u\in\mathcal{E}(X,\omega,m)$ and $u_j\to u$ in $\Capm$;
	
	ii) $H_m({u_j})\ll \Capm$ uniformly for $j\geq 1$ and $u_j\to u$ in $\Capm$;
	
	iii) $H_m({u_j})\ll \Capm$ uniformly for $j\geq 1$, $\varlimsup\limits_{j\to\infty} u_j\leq u$ and
	$$\lim\limits_{j\to\infty}\int\limits_{\{u_j<u- \delta\}}H_m(u_j)=0, \forall \delta >0.$$
\end{theorem}

\begin{proof} (i) $\Leftrightarrow$ (ii) It follows from Theorem \ref{th3.5}.
	
	(ii) $\Rightarrow$ (iii) This implication is trivial.
	
	iii) $\Rightarrow$ ii) By Proposition \ref{pro3.7}, we have $\lim\limits_{j\to\infty}\Capm (\{u_j<u-\delta\})=0$ for all $\delta >0$. So coupling this with $\varlimsup\limits_{j\to\infty} u_j\leq u$ we infer that  $u_j\to u$ in $\Capm$.
\end{proof} 
Recall that the weak convergence of Hessian measures does not imply the weak convergence of the corresponding functions (and reverse). Hence, such convergence problems should be added more suitable conditions. Usually results in this direction are called stability theorems in the literature. In \cite{CK06} good convergence properties were obtained in the domains of $\mathbb C^n$ under the assumption that all the Monge-Amp\`ere measures are dominated by a fixed measure vanishing on pluripolar sets. In the K\"ahler setting the corresponding problem was studied Xing in \cite{Xi09} in $\mathcal{E}^1(X,\omega)$. We will prove a similar result for the  $\mathcal{E}(X,\omega, m)$ which is the main result of the paper. 

\begin{theorem} Let $u_j,u\in\mathcal{E}(X,\omega,m)$. Assume that $H_m({u_j})\leq \mu$ for all $j\geq 1$ and $H_m(u)\leq \mu$ where $\mu$ is a positive measure satisfying $\mu(X)=\int_{X}\o^n=1$ and $\mu\ll \Capm$. Let also $\sup\limits_Xu_j=\sup\limits_Xu=0$. Then the following three statements are equivalent:
	
	i) $u_j\to u$ in $\Capm$;
	
	ii) $u_j\to u$ in $L^1(X,\omega)$;
	
	iii)  $H_m(u_j)\to H_m(u)$ weakly as $j\to\infty$.
\end{theorem}

\begin{proof}
	\n ii) $\Rightarrow$ i)  Since $u_j\to u$ in $L^1(X,\omega)$ we will prove that  
	\begin{equation}\label{eq3.19}\forall M\in{\mathbb R},  \lim_{j\to\infty}\int\limits_X|\max(u_j,M)-\max(u, M)|d\mu=0.
		\end{equation}
	Indeed, since $X$ is a compact K\"ahler manifold, we only need to prove the above equality on a local chart $U.$ Assume that $\o=dd^c\rho$ on $U$, where $\rho$ is a strictly plurisubharmonic function on $U$. Then we have $\max(u_j,M)+\rho$ and $\max(u,M)+\rho$ are $m$-subharmonic functions on $U$. We have 
	$$|\max(u_j,M)-\max(u, M)|=|(\max(u_j,M)+\rho)-(\max(u,M)+\rho) | .$$
	 Repeating the same arguments as in Proposition 3.10 in \cite{GZ05}, there exist $C\geq 1$ such that $$\frac{1}{C}\Capm(.)\leq \text{Cap}_{BT}(.)\leq C.\Capm(.),$$ where if $U$ is a local chart of $X$ and $E\subset U$ is a Borel subset of $U$ then
	$$\text{Cap}_{BT}(E,U):=\sup\Big\{\int_{E}(dd^cu)^{m}\wedge\omega^{n-m}: u\in SH_{m}(U), 0\leq u\leq 1\Big\}.$$
	Thus, from $\mu\ll\Capm$ it follows that $\mu\ll \text{Cap}_{BT}.$ Hence, we obtain that $\mu$ puts no mass on all $m$-polar sets on $U.$ \\
	\n Lemma 4.6 in \cite{PDjmaa} implies that 
	$$\lim_{j\to\infty}\int\limits_U[(\max(u_j,M)+\rho)-(\max(u, M)+\rho)]d\mu=0,$$
	and  equality \eqref{eq3.19} is proved. \\
	We claim that
	$$\begin{aligned}&\int\limits_{\{u_j<u-\delta\}}H_m(u_j)+\int\limits_{\{u<u_j-\delta\}}H_m(u)\\&\leq\int\limits_{\{u_j<M\}}d\mu+\int\limits_{\{u<M\}}d\mu+\frac1{\delta}\int\limits_{X}|\max(u_j,M)-\max(u, M)|d\mu.\hskip1cm(**)
	\end{aligned}$$
\n Indeed, we have 
\begin{equation}\label{e3.19}
\begin{aligned}
&	\int\limits_{\{u_j<u-\delta\}}H_m(u_j)=\int\limits_{\{u_j<u-\delta\}\cap\{u_j<M\}}H_m(u_j)+\int\limits_{\{u_j<u-\delta\}\cap\{u_j\geq M\}}H_m(u_j)\\
&\leq \int\limits_{\{u_j<M\}}d\mu +\int\limits_{\{u_j<u-\delta\}\cap\{u_j\geq M\}}d\mu\\
&\leq \int\limits_{\{u_j<M\}}d\mu + \int\limits_{\{u_j<u-\delta\}\cap\{u_j\geq M\}}\frac1{\delta}|\max(u_j,M)-\max(u, M)|d\mu.
	\end{aligned}
\end{equation}
Similarly, we also have
\begin{equation}\label{e3.20}
	\begin{aligned}
	&	\int\limits_{\{u<u_j-\delta\}}H_m(u)=\int\limits_{\{u<u_j-\delta\}\cap\{u<M\}}H_m(u)+\int\limits_{\{u<u_j-\delta\}\cap\{u\geq M\}}H_m(u)\\
	&\leq \int\limits_{\{u<M\}}d\mu +\int\limits_{\{u<u_j-\delta\}\cap\{u\geq M\}}d\mu\\
	&\leq \int\limits_{\{u<M\}}d\mu + \int\limits_{\{u<u_j-\delta\}\cap\{u\geq M\}}\frac1{\delta}|\max(u_j,M)-\max(u, M)|d\mu.
	\end{aligned}
\end{equation}
Note that $\Big( \{u_j<u-\delta\}\cap\{u_j\geq M\}\Big)\cap \Big(\{u<u_j-\delta\}\cap\{u\geq M\}\Big)=\emptyset.$ Thus, from inequality \eqref{e3.19} and inequality \eqref{e3.20} it follows that (**) is proved. Hence,
	if $M$\ is chosen sufficiently negative, by the assumption $\mu\ll \Capm$\ and Proposition \ref{md2.12} we infer that the first term and the second term on the right-hand side of (**) can be made arbitrarily small independently of $j$. Hence,  we get that 
	$$\lim_{j\to\infty}\Bigl\{\int\limits_{\{u_j<u-\delta\}}H_m(u_j)+\int\limits_{\{u<u_j-\delta\}}H_m(u)\Bigl\}=0.$$
	Moreover, since 
	 $H_m(u_j)\leq\mu, H_m(u)\leq\mu$ and $\mu\ll\Capm$, we imply  that $H_m(u_j)+H_m(u)\ll \Capm$.\\
	Thus, by Theorem \ref{2.2} we have $u_j\to u$\ in $\Capm$.
	
	\n ii) $\Rightarrow$ iii) 
	For any $\delta>0,$ we have
	\begin{align*}
	&\int_{\{u_j<u-\delta\}}H_m(u_j)\\
	&=\int_{\{u_j<u-\delta\}\cap [\{u<-k\}\cup\{u_j<-k\}] }H_m(u_j)+ \int_{\{u_j<u-\delta\}\cap\{u\geq -k\}\cap\{u_j\geq -k\}}H_m(u_j)\\
	&\leq \int_{\{u<-k\}\cup\{u_j<-k\}}\mu+{1\over \delta}\,\int_X|\max(u,-k)-\max(u_j,-k)|\,\mu.
	\end{align*}
	\n It follows from Proposition \ref{md2.12} that there exists $c>0$ such that $\Capm(u_j<-k)+\Capm(u<-k)\leq \frac{c}{k}$ for all $j$ and $k>0$. Since $\mu\ll \Capm$, we obtain that $\int_{\{u<-k\}\cup\{u_j<-k\}}\mu\longrightarrow 0$ as $k\to\infty$ uniformly for all $j$. So
	for any $\varepsilon>0$ there exists $k_\varepsilon>0$ such that 
	$$\int_{u_j<u-\delta}H_m({u_j})\leq  \varepsilon+{1\over \delta}\,\int_X|\max(u,-k_\varepsilon)-\max(u_j,-k_\varepsilon)|\,\mu, \;\; \forall j.$$
	Repeating the same argument as in inequality \eqref{eq3.19} we have 
	$$ \lim\limits_{j\to\infty}\int_X|\max(u,-k_\varepsilon)-\max(u_j,-k_\varepsilon)|\,d\mu=0.$$	
 Thus, we obtain
	$$\lim\limits_{j\to\infty}\int_{u_j<u-\delta}H_m(u_j)= 0.$$
\n 	By Proposition \ref{pro3.7} we get that \begin{equation}\label{e3.22}\lim\limits_{j\to\infty}\Capm(\{u_j<u-\delta \})=0.\end{equation}
\n Moreover, according to Proposition \ref{quasicontinuity}, there exists a open subset $U\subset X$ such that $\Capm(U)<\e$ and $u$ is continuous on $X\smallsetminus U.$
 Then  from $u_j\to u$ in $L^1(X,\omega)$ and by Theorem {\bf 3.2.13} in \cite{Ho94}( Hartog's Lemma) applying to the compact $K=X\setminus U$ it follows that there exists $j_0$ such that for $j\geq j_0$ we have $\{u_j> u+\delta\}\subset U.$ This means that we have 	\begin{equation}\label{e3.23}\Capm(\{u_j> u+\delta\})<\e, \end{equation}
 \n for $j\geq j_0$. 
Coupling equality \eqref{e3.22} and inequality \eqref{e3.23} we imply that $u_j$ converges to $u$ in $\Capm.$ By Theorem \ref{th3.1} we obtain
	$$H_m(u_j)\to H_m(u) \text{ weakly as } j\to\infty.$$
	
	\n  iii) $\Rightarrow$ ii)  Repeating the same argument as in Proposition 2.7 in \cite{GZ05}, from any subsequence of $\{u_j\}$, we can extract a further subsequence of the orginal subsequence that converges in $L^1$ to some function $v\in SH_m (X,\omega)$ and $\sup\limits_Xv=0$. Since the chosen convergent subsequence is arbitrary, it remains to show that $v=u.$ \\
	We claim now that $v$ belongs to $\mathcal{E}(X,\omega,m).$ For convenience of presentation, we can assume that $u_j\to v$ in $L^1(X,\omega)$ since the reasoning for the general case remains unchanged. Choosing a subsequence if necessary, we can assume that $u_j\to v$ almost everywhere in $X$ with respect to the smooth form $\omega^n$. Let $g_j=\max(u_j,u_{j+1},\dots ).$ Then its upper semicontinuous regularization $g_j^*$ satisfies that $0\geq g_j^*\geq u_j$ in $X$ and hence $g_j^*\in  {\cal E}(X,\omega,m)$. It is easy to see that $g^*_j\searrow v$ and hence $g^*_j\to v$ in $\Capm$ on $X$. According to Theorem \ref{th3.5}, we need to show that $H_m(g_j^*)\ll \Capm$ uniformly for $j\geq 1$ to get the claim $v\in\EcXo.$ Indeed, given $E\subset X$ and $k>0$, repeating the arguments as in the proof of Lemma \ref{lm3.2} we have
	
	$$\begin{aligned}\int_EH_m(g_j^*)&\leq 2^m\int_{\{u_j< -k+1\}}H_m(u_j)+2^mk^m\Capm(E)\\
		&\leq 2^m\mu\bigl(v<-k+2\bigr)+2^m\mu\bigl(|u_j-v|>1\bigr)+2^mk^m\Capm(E).\end{aligned}$$
\n We see that $\lim\limits_{k\to\infty}\mu\bigl(v<-k+2\bigr)=0.$ Then, for $\ve>0,$ there exists $k_0$ such that $2^m\mu\bigl(v<-k_0+2\bigr)\leq \frac{\ve}{3}.$ Moreover, $2^mk^m_0\Capm(E)\leq \frac{\ve}{3} $ when $\Capm(E)\leq \delta =\frac{\ve}{3.2^mk^m_0}.$\\
\n Thus, it remains to prove that there exists $j_0$ such that $\mu\bigl(|u_j-v|>1\bigr)\leq \ve_0=\frac{\ve}{3.2^m}$ for all $j\geq j_0$ to get $H_m(g_j^*)\ll \Capm$ uniformly for $j\geq 1$. Indeed, from $\Capm(u_j<-t)+\Capm(v<-t)\leq \frac{c}{t}$ and $\mu\ll \Capm$, there exists $t_0$ such that  $\mu(\{u_j<-t_0\})\leq \frac{\ve_0}{2} $ and $\mu(\{v<-t_0\})\leq \frac{\ve_0}{2}$. We will check that 
	\begin{equation*}\mu\bigl(|u_j-v|>1\bigr)\leq \mu\bigl(|\max(u_j,-t_0)-\max(v,-t_0)|>1\bigr)+\frac{\ve_0}{2}.\tag{***}\end{equation*}
	\n Obviously, inequality (***) is true on the sets $\Big[\{u_j\geq -t_0\}\cap \{v\geq -t_0\}\Big]\cup\Big[\{u_j< -t_0\}\cap \{v< -t_0\} \Big].$ On the set $\Big[\{u_j\geq -t_0\}\cap \{v< -t_0\}\Big],$ in equality (***) is equivalent to $\mu(\{u_j-v>1\})\leq\mu(\{u_j+t_0>1\})+\frac{\ve_0}{2}.$ This is true since we have inclusion $\{u_j-v>1\}\subset\Big[\{u_j+t_0>1\}\cup\{v< -t_0\} \Big].$ We use the same argument for the set $\Big[\{u_j< -t_0\}\cap \{v\geq -t_0\}\Big].$ Thus, inequality (***) is true.\\
	Therefore, we have
	$$\begin{aligned}\mu\bigl(|u_j-v|>1\bigr)&\leq\mu\bigl(|\max(u_j,-t_0)-\max(v,-t_0)|>1\bigr)+\frac{\ve_0}{2}\\
		&=\int_{\bigl(|\max(u_j,-t_0)-\max(v,-t_0)|>1\bigr)}d\mu+\frac{\ve_0}{2}\\
		&\leq \int_{\bigl(|\max(u_j,-t_0)-\max(v,-t_0)|>1\bigr)}|\max(u_j,-t_0)-\max(v,-t_0)|d\mu +\frac{\ve_0}{2}\\
	&\leq \int_{X}|\max(u_j,-t_0)-\max(v,-t_0)|d\mu +\frac{\ve_0}{2}.
	\end{aligned}$$
	 Repeating the argument as in the proof of equation \eqref{eq3.19} we have that
	 $$\lim\limits_{j\to\infty}\int_{X}|\max(u_j,-t_0)-\max(v,-t_0)|d\mu=0. $$
	 Hence, there exist $j_0$ such that for all $j\geq j_0$ we have
	 $$\int_{X}|\max(u_j,-t_0)-\max(v,-t_0)|d\mu \leq \frac{\ve_0}{2}. $$
	 Therefore, we have $\mu\bigl(|u_j-v|>1\bigr)\leq \ve_0$ for all $j\geq j_0$, which implies that $H_m(g_j^*)\ll \Capm$ and $v\in\EcXo.$\\
	\n Repeating the argument as in the proof  of the implication (ii) $\Rightarrow $ (i) we infer that $u_j\to v$ in $\Capm.$ According to Theorem \ref{th3.1} we imply that $H_m(u_j)$ converges weakly to $H_m(v)$ as $j\to\infty.$ Hence, we have $H_m(u)=H_m(v).$ By Theorem 4.1 in \cite{DC15} we deduce that $u=v.$
	
	\n i) $\Rightarrow$ ii) Since $u_j\longrightarrow u$ in $Cap_{m,\omega}$ and by Theorem \ref{th3.1} we have $H_m(u_j)\longrightarrow H_m(u)$ weakly as $j\to\infty$. Using the implication iii)$\Rightarrow$ ii) we get that $u_j\longrightarrow u$ in $L^1(X,\omega)$. Hence, the proof of Theorem 3.10 is complete. 
\end{proof}

\section*{Declarations}
\subsection*{Ethical Approval}
This declaration is not applicable.
\subsection*{Competing interests}
The authors have no conflicts of interest to declare that are relevant to the content of this article.
\subsection*{Authors' contributions }
Le Mau Hai and Nguyen Van Phu together studied  the manuscript.
%\subsection*{Funding }
%No funding was received for conducting this study.
\subsection*{Availability of data and materials}
This declaration is not applicable.

\end{document}